\documentclass[a4paper, 10pt, conference]{ieeeconf}      % Use this line for a4
\IEEEoverridecommandlockouts                              % This command is only
\usepackage{graphicx} % for pdf, bitmapped graphics files
\usepackage{amsmath} % assumes amsmath package installed
\usepackage{amssymb}  % assumes amsmath package installed
\usepackage{url}

\newtheorem{thm}{Theorem}
\newtheorem{proposition}[thm]{Proposition}

\newtheorem{remark}{Remark}
\newtheorem{definition}{Definition}
\newtheorem{assumption}{Assumption}

\title{
Periodic control laws for bilinear quantum systems with discrete spectrum
}

\author{\authorblockN{Nabile Boussa\"{i}d}
\authorblockA{Laboratoire de math\'ematiques\\
Universit\'e de Franche--Comt\'e\\
25030 Besan\c{c}on, France\\
{\tt\small Nabile.Boussaid@univ-fcomte.fr}}
\and
\authorblockN{Marco Caponigro}
\authorblockA{Institut \'Elie Cartan de Nancy and\\
INRIA Nancy Grand Est\\
54506 Vand{\oe}uvre, France\\
{\tt\small Marco.Caponigro@inria.fr}}
\and
\authorblockN{Thomas Chambrion}
\authorblockA{Institut \'Elie Cartan de Nancy and\\
INRIA Nancy Grand Est\\
54506 Vand{\oe}uvre, France\\
{\tt\small Thomas.Chambrion@inria.fr}}
}
\begin{document}

\maketitle
\thispagestyle{empty}
\pagestyle{empty}

%%%%%%%%%%%%%%%%%%%%%%%%%%%%%%%%%%%%%%%%%%%%%%%%%%%%%%%%%%%%%%%%%%%%%%%%%%%%%%%%
\begin{abstract}
We provide bounds on the error between dynamics of an infinite dimensional bilinear Schr\"{o}dinger equation 
and of its finite dimensional Galerkin approximations. 
Standard averaging methods are  used on the finite dimensional approximations 
to obtain constructive controllability results. As an illustration, the methods are applied on a model of a 2D rotating molecule.
\end{abstract}

%%%%%%%%%%%%%%%%%%%%%%%%%%%%%%%%%%%%%%%%%%%%%%%%%%%%%%%%%%%%%%%%%%%%%%%%%%%%%%%%
\section{INTRODUCTION}

\subsection{Physical context}

The state of a quantum system evolving on a finite dimensional Riemannian
manifold $\Omega$, with associated measure $\mu$, is described by its \emph{wave
function}, that is, a point in the unit sphere of $L^2(\Omega, \mathbf{C})$. A
system with wave function $\psi$ is in a subset $\omega$ of $\Omega$ with the
probability $\displaystyle{\int_{\omega} \!\!\!|\psi|^2 \mathrm{d}\mu}$.

 When submitted to an
excitation by an external  field ({\it e.g.} a laser) the time evolution of the wave 
function is governed by the bilinear Schr\"odinger
equation 
\begin{equation}\label{eq:blse}
\mathrm{i} \frac{\partial \psi}{\partial t}=-\frac{1}{2}\Delta \psi +V(x)
\psi(x,t) +u(t) W(x) \psi(x,t),
\end{equation}
where $V, W:\Omega\rightarrow \mathbf{R}$ are real functions describing  respectively the
physical properties of the uncontrolled system and the external field, and 
$u:\mathbf{R}\rightarrow \mathbf{R}$ is a real function of the time representing the intensity of the latter.

\subsection{Quantum control}

A natural question, with many practical implications, is whether there exists a
control $u$ that steers the quantum system  from a given initial position
to a given target.

Considerable efforts have been made to study the
controllability of~(\ref{eq:blse}). We refer to \cite{turinici,nersesyan,
beauchard-mirrahimi, mirrahimi-continuous, camillo,Schrod2} and references
therein for a description of the known theoretical results concerning the
existence of  controls steering a given source to a given target. As proved in
\cite{nersesyan, genericity-mario-paolo, genericity-mario-privat}, approximate
controllability is a generic property for systems of the type
(\ref{eq:blse}).

The main difficulty in the study of (\ref{eq:blse}) is the fact that the
natural state space, namely $L^2(\Omega,\mathbf{C})$, has infinite
dimension. To avoid difficulties when dealing with infinite dimensional systems,
for example when studying practical computations or simulations, one can project
system  (\ref{eq:blse}) on finite dimensional subspaces of
$L^2(\Omega,\mathbf{C})$. Obviously, a crucial issue is to guarantee that the finite
dimensional approximations have dynamics close to the one of
the original infinite dimensional system.

\subsection{Aim and content of the paper}

The contribution of this paper is twofold. First, in Section \ref{SEC_Weakly-coupled}, we provide an introduction to the class of \emph{weakly-coupled} bilinear systems (see Definition \ref{DEF_weakly_coupled}). A feature of these systems is that their dynamics is precisely approached by the dynamics of their Galerkin approximations (Proposition \ref{PRO_Good_Galerkin_approximtion}).
 In a second part, we apply
 general averaging theory for the approximate control of finite dimensional bilinear conservative systems using small amplitude periodic control laws.  The method is both very selective with respect to the frequency (which is a good point for quantum control) and extremely robust with respect to the shape of the control (Section \ref{SEC_finite_dimension}). Moreover, it provides easy and explicit estimates for the controllability time, the $L^1$ norm of the control and the error. Together with the results of Section \ref{SEC_Weakly-coupled}, this method provides a complete solution for the approximate control of infinite dimensional bilinear quantum systems with discrete spectrum and time estimates. As an illustration, we consider the rotation of a planar dipolar molecule in Section \ref{SEC_ex_rotation}.

%%%%%%%%%%%%%%%%%%%%%%%%%%%%%%%%%%%%%%%%%%%%%%%%%%%%%%%%%%%%%%%%%%%%%%%%%%%%%%%%
\section{WEAKLY-COUPLED BILINEAR SYSTEMS}\label{SEC_Weakly-coupled}

\subsection{Abstract framework}
We reformulate the problem \eqref{eq:blse} in a more abstract framework.  This
will allow us  to treat examples slightly more general than \eqref{eq:blse},
for instance, the example in~\cite[Section III.A]{weakly-coupled}. In a separable Hilbert
space $H$
endowed with norm $\| \cdot \|$ and Hilbert product $\langle \cdot, \cdot
\rangle$, we consider the evolution problem
\begin{equation}\label{eq:main}
\frac{d \psi}{dt}=(A+ u(t) B)\psi
\end{equation}
where $(A,B)$ satisfies the following assumption.

\begin{assumption}\label{ass:ass}
$(A,B)$ is a pair of linear operators such that
 \begin{enumerate}
 \item $A$ is skew-adjoint and has purely discrete spectrum $(-\mathrm{i}
\lambda_k)_{k \in \mathbf{N}}$, the sequence  $(\lambda_k)_{k \in \mathbf{N}}$
is positive non-decreasing and accumulates at
$+\infty$;\label{ASS_bounded_from_below}
	\item $B:H\rightarrow H$ is skew-adjoint and bounded.\label{ASS_B_bounded}
 \end{enumerate}
\end{assumption}

In the rest of our study, we denote by $(\phi_k)_{k \in
\mathbf{N}}$ an Hilbert basis of $H$ such
that $A\phi_k=-\mathrm{i}  \lambda_k \phi_k$ for every $k$ in $\mathbf{N}$.
We denote by $D(A+uB)$ the domain where
$A+uB$ is skew-adjoint.

Together with Kato-Rellich Theorem, the Assumption~\ref{ass:ass}.\ref{ASS_B_bounded} ensures that, for every
constants $u$ in $\mathbf{R}$, $A+u B$ is essentially
skew-adjoint on $D(A)$ and $\mathrm{i}(A+ u B )$ is bounded from below. Hence, for every initial condition $\psi_0$ in
$H$, 
for every
$u$ piecewise constant, $u:t\mapsto \sum_{j} u_j \chi_{(t_j,t_{j+1})}(t)$, with $0=t_0\leq t_1 \leq \ldots  \leq t_{N+1}$ and $u_0,\ldots,u_N$ in $\mathbf{R}$,
one can define
 the solution $t\mapsto \Upsilon^u_t\psi_0$ of \eqref{eq:main}   by
\begin{multline*}
 \Upsilon^u_t \psi_{0}=e^{(t-t_{j-1})(A + u_{j-1} B)}\circ\\\circ
e^{(t_{j-1}-t_{j-2})(A + u_{j-2} B)}\circ \cdots  \circ e^{t_{0}(A+
u_0  B)} \psi_{0},
\end{multline*}
for $t\in [t_{j-1},t_{j})$.
For a control $u$ in $L^{1}(\mathbf{R})$ we define the solution 
 using the 
 following classical continuity result.
 
\begin{proposition}\label{PRO_continuity_Upsilon}
Let $u$ and $(u_n)_{n\in \mathbf{N}}$ be in  $L^1(\mathbf{R})$. If for every $t$ in $\mathbf{R}$ $\int_0^t u_n(\tau) \mathrm{d} \tau$ converges to  $\int_0^t u(\tau) \mathrm{d} \tau$ as $n$ tends to infinity, then, for every $t$ in $\mathbf{R}$ and every $\psi_0$ in $H$, $(\Upsilon^{u_n}_t\psi_0 )_{n \in \mathbf{N}}$ converges to $\Upsilon^{u}_t\psi_0 $ as $n$ tends to infinity.
\end{proposition}

\subsection{Energy growth}
From Assumption~\ref{ass:ass}.\ref{ASS_bounded_from_below}, the operator
$\mathrm{i}A$ is self-adjoint with positive eigenvalues. For every $\psi$ in
$D(A)$, $\mathrm{i}A\psi=\sum_{j\in \mathbf{N}} \lambda_j \langle \phi_j, \psi
\rangle \phi_j$. For every $s\geq 0$, we define the linear operator
$|A|^s:=(\mathrm{i}A)^s$ by $|A|^{s}\psi=\sum_{j\in \mathbf{N}} \lambda_j^s
\langle \phi_j, \psi \rangle \phi_j$, for every $\psi$ in $D(|A|^s) = \{\psi \in
H\,:\,   \sum_{j\in \mathbf{N}} \lambda_j^{2s}
|\langle \phi_j, \psi \rangle|^{2}  < +\infty\}$.
We define the $s$-norm by $\|\psi \|_s=\||A|^s \psi\|$
for every $\psi$ in $D(|A|^s)$. The $1/2$-norm plays an important role in
physics; for every $\psi$ in $D(|A|^{1/2})$, the quantity $|\langle A \psi,\psi
\rangle |=\|\psi\|_{1/2}^2$ is the expected value of the energy.

\begin{definition}\label{DEF_weakly_coupled}
Let  $(A,B)$ satisfy
Assumption~\ref{ass:ass}. Then $(A,B)$ is
\emph{weakly-coupled}
if
%for every $u\in {\mathbf R}$, $D(|A+\sum_l u_lB_l|^{k/2})=D(|A|^{k/2})$ and
 there exists a constant $C$ such that, for every $\psi$ in $D(|A|)$, $ |\Re \langle |A|
\psi,B\psi \rangle |\leq C |\langle |A| \psi, \psi \rangle|$.
The coupling constant  $c(A,B)$ of system $(A,B)$ is the quantity
$
\sup_{\psi\in D(|A|)\setminus\{0\}}
%\frac
{ |\Re
\langle |A|
\psi,B\psi \rangle |}/{|\langle |A| \psi, \psi \rangle|}.
$
\end{definition}

The notion of weakly-coupled systems is closely related to the growth of the
expected value of the energy.

\begin{proposition}\label{PRO_croissance_norme_A} Let  $(A,B)$ satisfy Assumption~\ref{ass:ass} and be weakly-coupled.  Then,
for every $\psi_{0} \in D(|A|^{1/2})$, $K>0$,
 $T\geq 0$, and $u$ in $L^1([0,\infty))$  for which
$\|u\|_{L^1}< K$, one has
$\left\|\Upsilon^{u}_{T}(\psi_{0})\right\|_{1/2} <
e^{c(A,B) K} \| \psi_0 \|_{1/2}.$
\end{proposition}
\begin{proof}
We present here a simple proof  in the special case where $\psi_0$ belongs to $D(A^{2})$, $u$ is piecewise constant, and $D(|A|^{2})=D(|A+ u B|^{2})$ for every $u$ in $\mathbf{R}$. This last equality holds for the most common physical examples. A general proof of Proposition \ref{PRO_croissance_norme_A}, involving rather technical regularization techniques to relax this extra assumption is presented in \cite[Appendix]{weakly-coupled}.

First note that, for every $t\geq 0$, the set $D(|A|^{2})=D(|A+u B|^{2})$
is invariant for the unitary map $\psi\mapsto e^{t(A+u B)} \psi$. Moreover, for every $\psi$ in
$D(|A+ u B|^2)$, the mapping $t\mapsto |A+\sum u B |^2 e^{t(A+ u B)} \psi=
e^{t(A+ u B)} |A+ u B | \psi$ is $C^1$ from $[0,+\infty)$ to $H$, with derivative
$t\mapsto (A+ u B) e^{t(A+ u B )} |A+ u B | \psi = |A+ u B | (A+ u B ) e^{t(A+ u B )} \psi$. In other words, the mapping $t\mapsto e^{t(A+ u B )} \psi$ is $C^1$ from $[0,+\infty)$ to $D(|A+ u B |)=D(|A|)$.

Fix $u:[0,+\infty)\rightarrow \mathbf{R}$ piecewise constant, $\psi_0$ in $D(|A|^{k+1})$ and
consider the real mapping $f:t \mapsto \langle |A|^k \Upsilon^u_t \psi_0,\Upsilon^u_t \psi_0 \rangle $.
Since $\psi_{0}$ belongs to $D(|A+ u(t) B|^{k+1})$, then  $f$ is
absolutely continuous and for the argument above is piecewise $C^{1}$. For almost every $t$,
\begin{eqnarray*}
\frac{d}{dt} f(t) & = &  \frac{d}{dt}  \langle |A|^k \Upsilon^u_t\psi_0 , \Upsilon^u_t\psi_0 \rangle \\
%& =  \langle |A|^k \Upsilon^u_t\psi_0 ,  (A+\sum_{l=1}^p u_l(t) B_l) \Upsilon^u_t\psi_0 \rangle +
%\langle |A|^k  (A+\sum_{l=1}^p u_l(t) B_l)\Upsilon^u_t\psi_0 ,  \Upsilon^u_t\psi_0 \rangle \\
&= &2 \Re \langle |A| \Upsilon^u_t\psi_0 ,  (A+ u(t) B) \Upsilon^u_t\psi_0 \rangle\\
&= &2 u (t) \Re \langle |A| \Upsilon^u_t\psi_0 ,  B \Upsilon^u_t\psi_0 \rangle.
\end{eqnarray*}
Since $(A,B)$ is weakly-coupled, one has
\begin{eqnarray*}
|f'(t)| &\leq  &2   |u(t)| |\langle |A| \Upsilon^u_t\psi_0, B \Upsilon^u_t\psi_0
\rangle|\\
&\leq  &2 c(A,B) |u(t)| f(t).
\end{eqnarray*}
 From Gronwall's lemma, we get
% \begin{equation}%\label{EQ_majoration_norme_Ak}
 $$\langle |A| \psi(t), \psi(t) \rangle  \leq
e^{2 c(A,B) \int_{0}^t \!|u|(\tau) \mathrm{d}\tau} \|  \psi_0
\|_{1/2}^2$$ for every $\psi_0$ in $D(|A|^{2})$.
\end{proof}

\subsection{Good Galerkin approximation}
For every $N$ in $\mathbf{N}$, we define the orthogonal projection
$$
 \pi_N:\psi \in H \mapsto \sum_{j\leq N} \langle \phi_j,\psi\rangle
\phi_j \in H.
$$
\begin{proposition}\label{PRO_troncature}
Let $(A,B)$ satisfy
Assumption~\ref{ass:ass}, and be weakly-coupled. Then,
for every  $n\in \mathbf{N}$, $N \in \mathbf{N}$,
$(\psi_j)_{1\leq j \leq n}$ in $D(|A|^{1/2})^n$,
and for every $L^{1}$ function $u$,
\begin{equation}\label{eq:feps2}
\|(\mathrm{Id} - \pi_{N})
\Upsilon^{u}_{t}(\psi_{j})\|<  \frac{e^{c(A,B) \|u\|_{L^1} }\|\psi_j\|_{1/2}}{\sqrt{\lambda_{N+1}}}.
\end{equation}
for every $t \geq 0$ and $j=1,\ldots,n$.
\end{proposition}

\begin{proof}
Fix $j \in \{1,\ldots,n\}$. For every $N > 1$, one has
\begin{eqnarray*}\label{eq:estimates}
\left\|(\mathrm{Id} - \pi_{N})
\Upsilon^{u}_{t}(\psi_{j})\right\|^2&=&
\sum_{n = N+1}^{\infty}  | \langle  \phi_{n},
\Upsilon^{u}_{t}(\psi_{j}) \rangle|^{2}\\
 &\leq&
\lambda_{N+1}^{-1}\left\|\Upsilon^{u}_{t}(\psi_{j})\right\|_{1/2}^2.
\end{eqnarray*}
By Proposition~\ref{PRO_croissance_norme_A}, for
every $t>0$,
$\left\|\Upsilon^{u}_{t}(\psi_{j})\right\|_{1/2}^2 \leq
e^{2c_{k}(A,B) \|u\|_{L^1}}\| \psi_j
\|_{1/2}^2$.
The conclusion then follows by Proposition~\ref{PRO_continuity_Upsilon}.
%Equation
%\eqref{eq:feps2} follows as, for every $l=1,\ldots,p$,
%$\|B_l\psi \|\leq d \| |A|^{\frac{r}{2}}\psi \|$.
\end{proof}

\begin{remark}
 Since $B$ is bounded, then $\left\|B(\mathrm{Id} - \pi_{N})
\Upsilon^{u}_{t}(\psi_{j})\right\|$
tends to $0$ as $N$ tends to infinity  uniformly with
respect to $u$ of $L^{1}$-norm smaller than a given constant.
\end{remark}

\begin{definition}\label{DEF_Galerkin_approx}
Let $N \in \mathbf{N}$.  The \emph{Galerkin approximation}  of \eqref{eq:main}
of order $N$ is the system in $H$
\begin{equation}\label{eq:sigma}
\dot x = (A^{(N)} + u(t) B^{(N)}) x \tag{$\Sigma_{N}$}
\end{equation}
where $A^{(N)}=\pi_N A \pi_N$ and $B^{(N)}=\pi_N B_{l} \pi_N$ are the
\emph{compressions} of $A$ and $B$ (respectively).
\end{definition}

{We denote by $X^{u}_{(N)}(t,s)$ the propagator of \eqref{eq:sigma}
for a $L^{1}$ function $u$.}

\begin{remark}
The operators $A^{(N)}$ and $B^{(N)}$ are defined
on the \emph{infinite} dimensional space $H$. However, they have finite rank and
the dynamics of $(\Sigma_N)$ leaves invariant the $N$-dimensional space
$\mathcal{L}_{N} = \mathrm{span}_{1\leq l \leq N} \{\phi_l\}$. Thus,
$(\Sigma_N)$ can be seen as a finite dimensional bilinear system in
$\mathcal{L}_{N}.$
\end{remark}

\begin{proposition}[Good Galerkin Approximation]
	\label{PRO_Good_Galerkin_approximtion}
Let $(A,B)$ satisfy Assumption~\ref{ass:ass} and be weakly-coupled. Then
for every $\varepsilon > 0 $, $K\geq 0$, $n\in \mathbf{N}$, and
$(\psi_j)_{1\leq j \leq n}$ in $D(|A|^{1/2})^n$
there exists $N \in \mathbf{N}$
such that
for every $L^{1}$ function $u$
$$%\begin{equation}
\|u\|_{L^{1}} < K \implies \| \Upsilon^{u}_{t}(\psi_{j}) -
X^{u}_{(N)}(t,0)\pi_{N} \psi_{j}\| < \varepsilon,
$$%\end{equation}
for every $t \geq 0$ and $j=1,\ldots,n$.
\end{proposition}

\begin{proof}
Fix $j$ in $\{1,\ldots,n\}$ and consider
the map $t\mapsto
\pi_{N} \Upsilon^{u}_{t}(\psi_{j})$ that is absolutely continuous and satisfies,
for almost every $t \geq 0$,
\begin{multline}
\frac{d}{dt} \pi_{N} \Upsilon^{u}_{t}(\psi_{j}) = (A^{(N)} +  u(t)B^{(N)}) \pi_{N} \Upsilon^{u}_{t}(\psi_{j}) \\+ u(t) \pi_{N} B (\mathrm{Id} - \pi_{N}) \Upsilon^{u}_{t}(\psi_{j}).
\end{multline}
Hence, by variation of constants, for every $t \geq 0$,
\begin{multline}\label{EQ_preuve_good_Galerkin}
\pi_{N} \Upsilon^{u}_{t}(\psi_{j})= X^{u}_{(N)}(t,0)  \pi_{N}\psi_{j} \\+
 \int_{0}^{t} \!\!\!
X^{u}_{(N)}(t,s) \pi_{N} B(\mathrm{Id} - \pi_{N}) \Upsilon^{u}_{s}(\psi_{j})
u(\tau)  \mathrm{d}\tau.
\end{multline}
By Proposition~\ref{PRO_troncature}, the norm of $t \mapsto B(\mathrm{Id} -
\pi_{N}) \Upsilon^{u}_{t}(\psi_{j})$ is
less than
$\|B\|\lambda_{N+1}^{-1/2}e^{c(A,B)K}\|\psi_{j}\|_
{1/2}$. Since $X^{u}_{(N)}(t,s) $ is unitary,
$
\|\pi_{N} \Upsilon^{u}_{t}(\psi_{j}) -X^{u}_{(N)}(t,0) \pi_{N}\psi_{j}\| <
\|u\|_{L^1}
\|B\| \lambda_{N+1}^{-1/2}e^{c(A,B)K}\|\psi_{j}\|_{k/2}.
$
Then
\begin{eqnarray*}
\lefteqn{\|\Upsilon^u_t(\psi_j) - X^u_{(N)}(t,0)\pi_{N} \psi_{j}\|
 }\\
 &\leq & \|(\mathrm{Id} - \pi_{N}) \Upsilon^{u}_{t}(\psi_{j})\|\\
&&+ \|\pi_{N} \Upsilon^{u}_{t}(\psi_{j}) - X^{u}_{(N)}(t,0)  \pi_{N} \psi_{j}
\| \nonumber\\
& \leq& \lambda_{N+1}^{-1/2} (1+K \|B\|) e^{c(A,B)K}\|\psi_{j}\|_{1/2}. \label{eq:qwer}
\end{eqnarray*}
This completes the proof  since $\lambda_N$ tends to infinity as $N$
goes to infinity. \end{proof}

%\addtolength{\textheight}{-3cm}   % This command serves to balance the column lengths

\section{Control of finite dimensional  conservative bilinear systems}\label{SEC_finite_dimension}

\subsection{Averaging results}\label{SEC_averaging}
In this Section, we focus on the control system $(\Sigma_N)$ of Definition \ref{DEF_Galerkin_approx}.
The matrix $A^{(N)}$ is diagonal with eigenvalues $(-\mathrm{i} \lambda_j)_{1 \leq j \leq N}$ with $\lambda_1 \leq \lambda_2 \leq \ldots \leq \lambda_N$.  We denote with $(b_{jk})_{1\leq j,k\leq N}$ the entries of $B^{(N)}$.

\begin{thm}\label{PRO_main_result}
 Let $j,k$ be two integers such that $1\leq j<k\leq N$. Assume that $b_{jk}\neq 0$ and that for every $l,m\leq N$, $|\lambda_l-\lambda_m| \in \mathbf{N} |\lambda_j-\lambda_k|$ implies $\{j,k\}=\{l,m\}$ or $b_{lm}=0$. Define $T=\frac{2\pi}{\lambda_k-\lambda_j}$. Then,  for every $u^{\ast}$ in $L^{1}(\mathbf{R})$ satisfying
 $$
 \int_0^T \!\!\! u^{\ast}(\tau) e^{\mathrm{i} (\lambda_k-\lambda_j) \tau} \mathrm{d}\tau \neq 0
 \mbox{ and}
 \int_0^T \!\!\!u^{\ast}(\tau) e^{\mathrm{i} (\lambda_l-\lambda_m) \tau} \mathrm{d}\tau = 0,
 $$ 
 for every $l,m$ such that $|\lambda_j-\lambda_k|\in  |\lambda_l-\lambda_m| \mathbf{N}$ and $b_{lm}\neq 0$, one has
 %$$
%  n \left (1- |\langle \phi_k, X^{\frac{u^{\ast}}{n}}({nT^{\ast}},0)\phi_j \rangle | \right )  \leq  \frac{ C\int_0^T |u^{\ast}(\tau)|\mathrm{d}\tau }{\left |\int_0^T u^{\ast}(\tau)e^{\lambda_j-\lambda_k) \tau} \mathrm{d}\tau \right |}
% $$
 \begin{eqnarray*}
\lefteqn{1- |\langle \phi_k, X^{\frac{u^{\ast}}{n}}({nT^{\ast}},0)\phi_j \rangle | }
 \\
  &\leq&  \frac{C}{n} \frac{\pi} {2  |b_{jk}|} \frac{\left (\int_0^T\!\! |u^{\ast}(\tau)|\mathrm{d}\tau \right )^2 }{ \left |\int_0^T u^{\ast}(\tau)e^{\mathrm{i}(\lambda_j-\lambda_k) \tau} \mathrm{d}\tau \right |}
 \end{eqnarray*}
 with
 $$
 C=\frac{1+\|A^{(N)}\|+\|B^{(N)}\|}{\inf_{\frac{\lambda_l-\lambda_m}{\lambda_j-\lambda_k}\notin \mathbf{Z}} \left |\sin \left ( 2 \pi \frac{\lambda_l-\lambda_m}{\lambda_j-\lambda_k} \right ) \right |}
 $$
  and
 $$T^{\ast}=\frac{\pi  T}{2 |b_{jk}|  \left |\int_0^T \!\! u^{\ast}(\tau)e^{\mathrm{i}(\lambda_{j}-\lambda_{k}) \tau}  \mathrm{d}\tau \right |}.$$
\end{thm}

Theorem \ref{PRO_main_result} above states that $|\langle \phi_k, X^{\frac{u^{\ast}}{n}}({nT^{\ast}},0)\phi_j \rangle |$ tends to one as $n$ tends to infinity.
This convergence result may be obtained using classical averaging theory (see for instance \cite{Sanders}). Another  proof based on the particular algebraic structure of the system is given in \cite{periodic}. The estimates given here are a consequence of \cite[Lemma 8]{periodic}.

\subsection{Efficiency of the transfer}\label{SEC_Efficiency}

Using the notations of the last paragraph, for every periodic function $u^{\ast}$ with period $T=\frac{2\pi}{|\lambda_j-\lambda_k|}$, we define the efficiency of $u^{\ast}$ with respect to the transition $(j,k)$ as the real quantity:
\begin{eqnarray*}
E^{(j,k)}(u^{\ast})&=&
\frac{\left |\int_0^{T}\! u^{\ast}(\tau) e^{\mathrm{i}(\lambda_j-\lambda_k) \tau} \mathrm{d}{\tau} \right |}{\int_0^{T}\! |u^{\ast}(\tau) | \mathrm{d}{\tau} } .
\end{eqnarray*}
For every $u$, $0\leq E^{(j,k)}(u) \leq 1$.
For every $\{j, k\}$, $\sup_u  E^{(j,k)}(u)=1$ (consider a sequence of  controls that tends to a periodic sum of Dirac functions). An example of $u^{\ast}$ with zero efficiency is presented in Section \ref{SEC_ex_rotation}.

An intuitive explanation of the efficiency could be the following: asymptotically, the $L^1$-norm of the control  needed to induce the transition between levels $j$ and $k$ using periodic controls of the form $u_n$ is equal to
$\pi/(2|b_{jk}| E^{(j,k)}(u^{\ast}))$.

\subsection{Design of control laws}\label{SEC_Design}

The system \eqref{eq:main} being given, the design of an effective control law fulfilling the hypotheses of Theorem~\ref{PRO_main_result}  is an important
 practical issue.
 To generate  a transfer from level $j$ to level $k$, one should chose a control $u$ such that $E^{(j,k)}(u)$ be as large as possible and $E^{(l,m)}(u)$ be zero (or arbitrarily close to zero) for every $l,m$ such that $\lambda_{l}-\lambda_{m} \in (\lambda_j-\lambda_k) \mathbf{Z}$.

 A natural choice for the control law can be
 $
 t\mapsto \cos \left ( {|\lambda_j-\lambda_k|} t \right ).
 $
 The efficiency with respect to transition $(j,k)$ is equal to $\pi/4\approx 0.79$. All other efficiencies are zeros.

 This natural cosine control law may however not match every desirable properties of the control law. For instance, if for some reason the control law has to be positive, one could chose $
 t\mapsto 1+ \cos \left ( {|\lambda_j-\lambda_k|} t \right ).
 $
 The efficiency with respect to transition $(j,k)$ is equal to $1/2$. All other efficiencies are zeros.

 If there are no resonance, that is if $\lambda_j-\lambda_k \notin (\lambda_l-\lambda_m) \mathbf{Z} $ for every $\{l,m\}\neq \{j,k\}$, one could consider the choice of a periodic Dirac pulse $t\mapsto \sum_{n \in \mathbf{Z}} \delta_{n 2 \pi |\lambda_j-\lambda_k|}(t)$, whose efficiency with respect to transition $(j,k)$ is equal to one. If one desires to avoid the use of distributional control laws (Equation \eqref{eq:main} has then to be understood in the measure sense), one may consider standard $L^1$ function close enough to Dirac pulses. Examples are presented in Section \ref{SEC_ex_rotation}.

 Finally, let us mention the algorithm described in \cite{Schrod2} allows to build $u^{\ast}$ positive and piecewise constant. If $\lambda_j-\lambda_k=1$ and the only resonances to be considered are such that $|\lambda_l-\lambda_m| \in \{a_1,a_2,\ldots,a_p\}$, then the efficiency of $u^{\ast}$ with respect to transition $(j,k)$ is $$ \prod_{n=1}^p \cos \left (\frac{\pi}{2 a_n} \right ).$$
   In the worst case, this algorithm guarantees  $E^{(j,k)}(u)\geq \prod_{n \geq 2} \cos (\pi/2n) >0.43$, what may seem poor with respect to the cosine law.  However, this algorithm is especially useful to handle the case of  high order resonances. Indeed, if $a_1,a_2,\ldots,a_p$ are all greater than $N$, then the efficiency with respect to transition $(j,k)$ is greater than $$\exp \left (-\frac{\pi^2}{4N} -\frac{\pi^4}{48 N^3}\right ),$$ which tends to one as $N$ tends to infinity.

\section{Rotation of a planar molecule}\label{SEC_ex_rotation}

In this Section, we apply our results to the well studied example of the rotation of a planar molecule (see, for instance, \cite{salomon-HCN,noiesugny-CDC, Schrod2}).
\subsection{Presentation of the model}

We consider a linear molecule with fixed length and center of mass. We assume that the molecule is constrained to stay in a fixed plane and that its only degree of freedom is the rotation, in the plane, around its center of mass. The state of the system at time $t$ is described by a point $\theta \mapsto \psi(t,\theta)$ of $L^2(\Omega,\mathbf{C})$ where $\Omega=\mathbf{R}/2\pi \mathbf{Z}$ is the one dimensional torus.
The Schr\"odinger equation writes
\begin{equation}\label{EQ_Schrod_circle}
\mathrm{i} \frac{\partial \psi}{\partial t}(t,\theta)= -\Delta \psi(t,\theta) + u(t) \cos(\theta) \psi(t,\theta),
\end{equation}
where $\Delta$ is the Laplace-Beltrami operator on $\Omega$.
The  self-adjoint operator $-\Delta$ has purely discrete spectrum $\{k^2,k \in \mathbf{N}\}$. All its eigenvalues are double but zero which is simple. The eigenvalue zero is associated with the constant
functions. The eigenvalue $k^2$ for $k>0$ is associated with the two eigenfunctions $\theta \mapsto\frac{1}{\sqrt{\pi}} \cos(k \theta)$ and $\theta \mapsto \frac{1}{\sqrt{\pi}} \sin(k \theta)$. The Hilbert space $H=L^2(\Omega,\mathbf{C})$ splits in two subspaces $H_e$ and $H_o$, the spaces of even and odd functions of $H$ respectively. The spaces $H_e$ and $H_o$ are stable under the dynamics of (\ref{EQ_Schrod_circle}), hence no global controllability is to be expected in $H$.

\subsection{Non-resonant case}\label{SEC_strongly}

We first focus on the space $H_o$. The restriction $A$ of $\mathrm{i} \Delta$ to $H_o$ is skew adjoint, with simple eigenvalues $(-\mathrm{i}k^2)_{k\in \mathbf{N}}$ associated to the eigenvectors 
$$
\left (\phi_k:\theta \mapsto \frac{1}{\sqrt{\pi}} \sin(k \theta) \right )_{k\in \mathbf{N}}.
$$ 
The restriction $B$ of $\psi\mapsto -\mathrm{i} \cos(\theta) \psi$ to $H_o$ is skew-adjoint and bounded. The pair $(A,B)$ satisfies Assumption \ref{ass:ass} and is weakly-coupled (see \cite[Section~III.C]{weakly-coupled}).

The Galerkin approximations of $A$ and $B$ at order $N$ are
$$
A^{(N)}= - \left ( \begin{array}{cccc}
                \mathrm{i} & 0 & \cdots & 0 \\
		0 	& 4 \mathrm{i} & \ddots & \vdots \\
		\vdots & \ddots & \ddots & 0\\
		0 & \cdots & 0 & N^2 \mathrm{i}
                \end{array}
\right )  \mbox{ ~~and }$$
$$
B^{(N)}= - \mathrm{i}\left ( \begin{array}{ccccc}
                0 & 1/2 & 0 & \cdots & 0 \\
		 1/2 & 0 & 1/2 & \ddots & \vdots \\
		0  & \ddots & 0  &\ddots  & 0\\
		\vdots & \ddots  & 1/2 & 0 & 1/2\\
		0 & \cdots & 0 & 1/2 & 0
                \end{array}
\right ).
$$

Our aim is to transfer the wave function from the first eigenspace to the second one.
The numerical simulation will be done on some finite dimensional space $\mathbf{C}^N$. The controls we will use in the following have $L^1$ norm less than $13/3$ and, from Proposition \ref{PRO_croissance_norme_A}, the $|A|^{\frac{1}{2}}$ norm of $\Upsilon^u_t(\phi_1)$ will remain less than $\exp(13/2)\approx 665$ for all time. From \cite[Remark 4]{weakly-coupled},
the error made when replacing the original system by its Galerkin approximation  of order $\sqrt{13 e^{13/2}/3 /10^{-2}}\approx 288228$  is smaller than $\varepsilon=10^{-2}$. This estimate is indeed very conservative and it can be improved using the regularity of the operator $B$.

From \cite[Section~IV.C]{weakly-coupled}, for every integer $l$, for every $t$ in $[0,+\infty)$, for every locally integrable control $u$ (not necessarily periodic),
$$ 
|\langle \phi_{k+1}, \Upsilon^{u}_t\phi_1 \rangle | \leq \frac{1}{k!} \left ( \int_0^t \!\!\!|u(\tau)|\mathrm{d}\tau   \right )^{k}.
$$
%$$|\langle \phi_k, \Upsilon^{u}_t\phi_1 \rangle | \leq \sum_{l=0}^{k-1} \frac{|\langle \phi_l, B^l\phi_1 |}{l!} \left ( \int_0^t \!\!\!|u(\tau)|\mathrm{d}\tau   \right )^l + \frac{\|B^k \phi_k\|}{k!} \left ( \int_0^t \!\!\!|u(\tau)|\mathrm{d}\tau   \right )^k.$$
As a consequence, if $\|u\|_{L^1}\leq 13/3$, then $\|\pi_{22} B(\mathrm{Id}-\pi_{22})\Upsilon_t^u (\phi_1)\|\leq 5.10^{-7}$ for every $t$ in $[0,+\infty)$. Using this inequality, one gets that the error made when replacing the original system by its Galerkin approximation  of order $22$  is smaller than $\varepsilon=3.10^{-6}$ when  $\|u\|_{L_1}\leq 13/3$.

The transition between the levels $1$ and $2$ is resonant, indeed, $5^2-4^2=9=3(2^2-1^2)$. Nevertheless, for every $\{l_1,l_2\}\neq\{1,2\}$ such that $\lambda_{l_1}-\lambda_{l_2} \in 3 \mathbf{Z}$ and $\langle \phi_{l_1},B\phi_{l_2}\rangle \neq 0$, one has $l_1>2$ and $l_2>2$. Hence, for every $\frac{2\pi}{3}$-periodic function $u$, the limit of the propagator $X^{u/n}_{(N)}(t,0)$  leaves invariant the subspace generated by $\phi_1$ and $\phi_2$ and the result of Theorem \ref{PRO_main_result} applies (without having to check that all efficiencies of $u$ for the transition $(l_1,l_2)$ with $l_1-l_2\in {3}(\mathbf{Z}\setminus\{1\})$ are zero).

We illustrate the notion of efficiency on some examples of control, namely $u^{\ast}:t\mapsto \cos^l(3t)$ for $l\in \{1,2,3,4,5\}$.

The efficiency is zero when $l$ is even. In numerical simulations, the quantity $|\langle \phi_2,X^{u^{\ast}}_{(22)}(t,0)\phi_1 \rangle |$ is less than $2.10^{-5}$ for every $t<500$ (see Figure 1 for $l=2$).

When $l$ is odd, the efficiency is not zero. To estimate numerically the efficiency, one considers, for  $n\in \{1,10,30\}$, the
first maximum $p^{\dag}$  of $t\mapsto |\langle \phi_2,X^{u^{\ast}/n}_{(N)}(t,0)\phi_1 \rangle |$, reached at time $t^{\dag}$, and computes
$$\frac{(1-p^{\dag})n\pi}{2|\langle \phi_1, B \phi_2\rangle | \int_0^{t^{\dag}}\!|u^{\ast}(\tau)| \mathrm{d}\tau}.$$
The \textbf{\texttt{Scilab}} source codes used for the simulation are available on the web page
\cite{webpage}. We sum up the results in Table 1.
\begin{table}
\renewcommand{\arraystretch}{1.3}
\caption{Numerical Efficiencies of some periodic shapes}
\label{table_example}\begin{tabular}{|c||c|c|c|c|}
\hline
Control $u^{\ast}$ &  $n$ & Time $t^{\dag}$ & Precision  & Numerical \\
(Efficiency) & &                  &   $1-p^{\dag}$          & Efficiency\\
\hline
 & $n=1$ & $6.8$ & $2. 10^{-2}$ & 73\% \\ \cline{2-5}
$t\mapsto \cos(3t)$
 & $n=10$   & 63  & $4.10^{-4}$ & 78\% \\ \cline{2-5}
$\pi/4\approx 79\%$  & $n=30$   & 189 & $3.10^{-5}$ & 78\%  \\
\hline
& $n=1$ & $8.9$ & $2. 10^{-2}$ & 83\% \\ \cline{2-5}
$t\mapsto \cos(3t)^3$
 & $n=10$   & 84  & $2.10^{-4}$ & 88\% \\ \cline{2-5}
$9 \pi/32 \approx 88\%$  & $n=30$   & 252 & $2.10^{-5}$ & 88\%  \\
\hline
& $n=1$ & $10$ & $7. 10^{-3}$ & 93\% \\ \cline{2-5}
$t\mapsto \cos(3t)^5$
 & $n=10$   & 101  & $2.10^{-4}$ & 92\% \\ \cline{2-5}
$75 \pi/256\approx 92\%$  & $n=30$   & 302 & $2.10^{-5}$ & 92\%  \\
\hline
\end{tabular}
%
%\begin{tabular}{|c||c|c|c|}
%\hline
%Control $u^{\ast}$ &  $n$ & Time $t^{\dag}$ & Precision   \\
%(Efficiency) & &                  &   $1-p^{\dag}$        \\
%\hline
% & $n=1$ & $6.8$ & $2. 10^{-2}$ \\ \cline{2-4}
%$t\mapsto \cos(3t)$
% & $n=10$   & 63  & $4.10^{-4}$  \\ \cline{2-4}
%$\pi/4\approx 79\%$  & $n=30$   & 189 & $3.10^{-5}$  \\
%\hline
%& $n=1$ & $8.9$ & $2. 10^{-2}$  \\ \cline{2-4}
%$t\mapsto \cos(3t)^3$
% & $n=10$   & 84  & $2.10^{-4}$  \\ \cline{2-4}
%$9 \pi/32 \approx 88\%$  & $n=30$   & 252 & $2.10^{-5}$   \\
%\hline
%& $n=1$ & $10$ & $7. 10^{-3}$  \\ \cline{2-4}
%$t\mapsto \cos(3t)^5$
% & $n=10$   & 101  & $2.10^{-4}$  \\ \cline{2-4}
%$75 \pi/256\approx 92\%$  & $n=30$   & 302 & $2.10^{-5}$   \\
%\hline
%\end{tabular}
\end{table}

\begin{figure}[thpb]
 \centering
 \includegraphics[width=8.6cm]{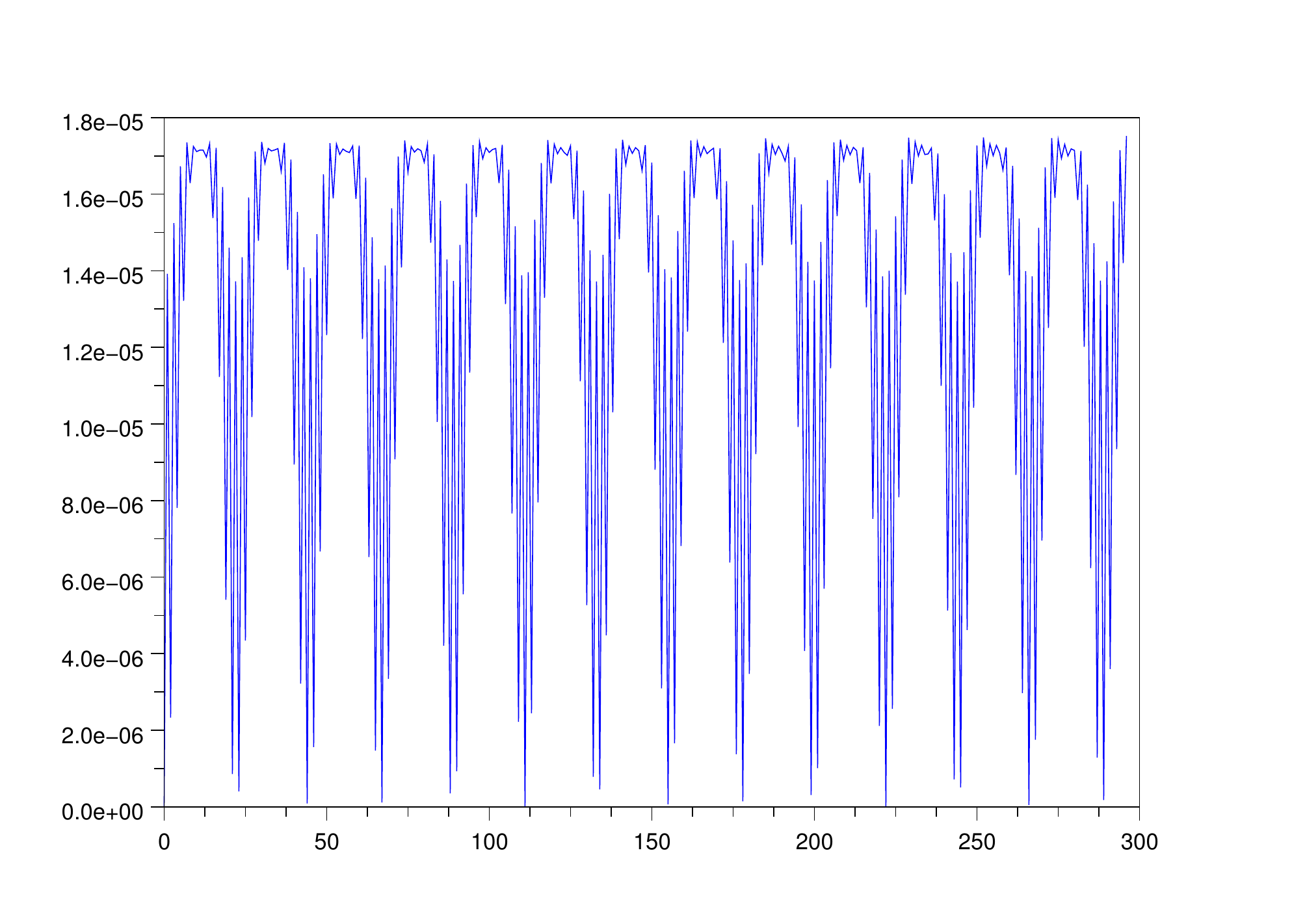}
 \caption{Evolution of the square of the modulus of the second coordinate when applying the control $:t\mapsto \cos^2(3t)/30$ on the planar molecule (odd subspace) with initial condition $\phi_1$. The simulation has been done on a Galerkin approximation of size $N=22$.}
\end{figure}

 \begin{figure}[thpb]
 \centering
\includegraphics[width=8.6cm]{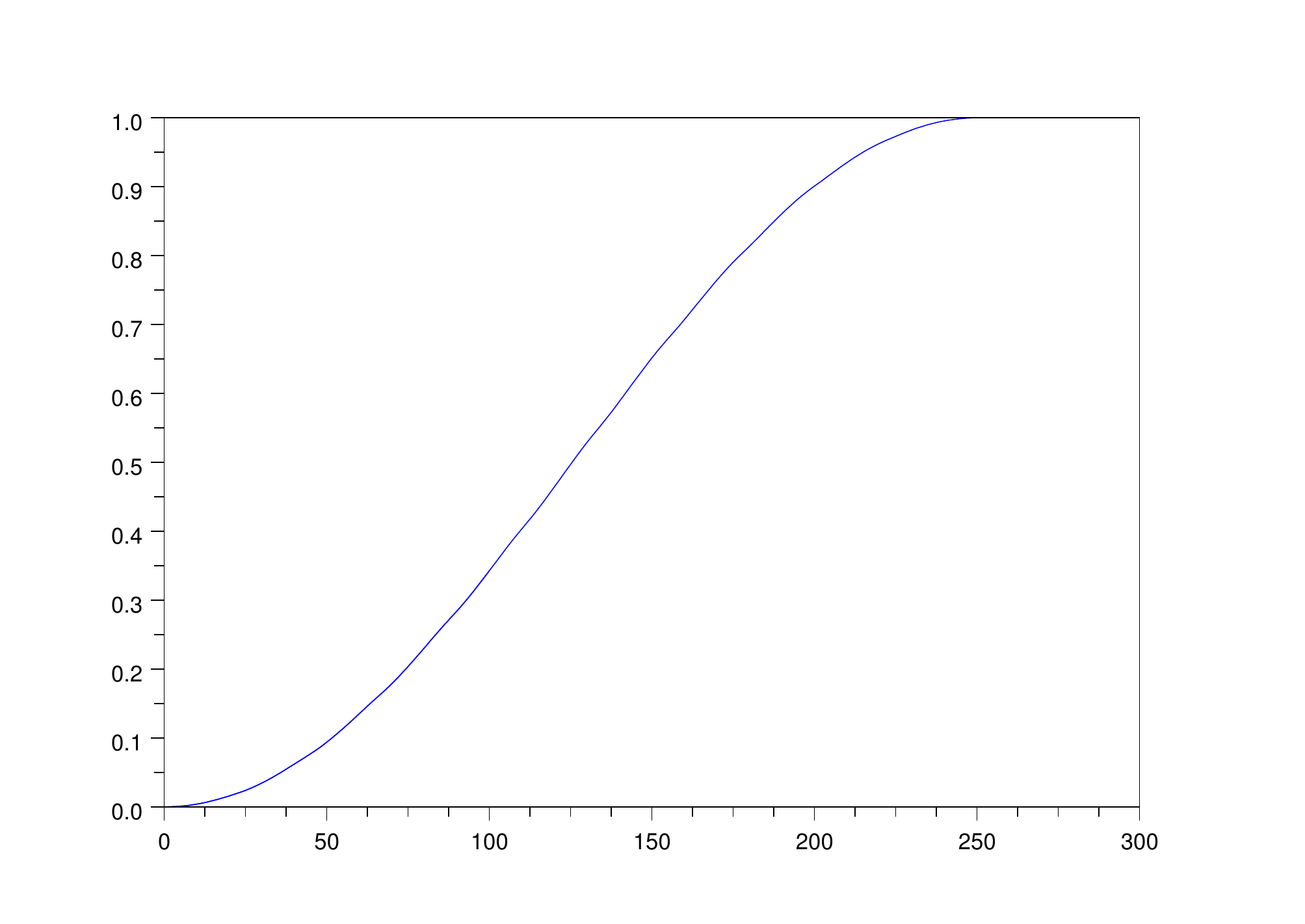}
\caption{Evolution of the  square of the modulus of the second coordinate when applying the control $t\mapsto \cos^3(3t)/30$  on the planar molecule (odd subspace) with initial condition $\phi_1$. The simulation has been done on a Galerkin approximation of size $N=22$.}
\end{figure}

\subsection{Resonant case}\label{SEC_not_strongly}
We focus on the space $H_e$. The restriction $A$ of $\mathrm{i} \Delta$ to $H_e$ is skew adjoint, with simple eigenvalues $(-\mathrm{i}k^2)_{k\in \mathbf{N} \cup \{0\}}$ associated to the eigenvectors $(\phi_k)_{k\in \mathbf{N} \cup \{0\}}$, with $\phi_k:\theta  \mapsto \frac{1}{\sqrt{\pi}} \cos(k \theta)$ for $k$ in $\mathbf{N}$ and $\phi_0: \theta \mapsto \frac{1}{\sqrt{2\pi}}$. The restriction $B$ of $\psi\mapsto -\mathrm{i}\cos(\theta) \psi$ to $H_e$ is skew-symmetric. The pair $(A+\mathrm{i},B)$ satisfies Assumption \ref{ass:ass}. The translation from $A$ to $A+\mathrm{i}$ induces just  a phase shift and will be neglected in the following.

The Galerkin approximation of $A$ and $B$ at order $N$ are
$$
A^{(N)}= - \left ( \begin{array}{cccc}
                0 & 0 & \cdots & 0 \\
		0 	&   \mathrm{i} & \ddots & \vdots \\
		\vdots & \ddots & \ddots & 0\\
		0 & \cdots & 0 & (N-1)^2 \mathrm{i}
                \end{array}
\right ) \mbox{~~and}$$
$$
B^{(N)}= - \mathrm{i}\left ( \begin{array}{ccccc}
                0 & 1/\sqrt{2} & 0 & \cdots & 0 \\
		1/\sqrt{2} & 0 &1/2 & \ddots & \vdots \\
		0  & \ddots & 0  &\ddots  & 0\\
		\vdots & \ddots  & 1/2 & 0 & 1/2\\
		0 & \cdots & 0 & 1/2 & 0
                \end{array}
\right ).
$$

Our aim is to transfer the population from the first eigenspace, associated with eigenvalue $0$, to the second one, associated with eigenvalue $\mathrm{i}$. The transition $(1,2)$ is resonant (indeed $2^2-1^2=3=3(1^2-0^2)$), and 
unlike what happens on the space of odd eigenfunctions, the limit matrix $M^{\dag}$ does not necessarily stabilize the space spanned by $\phi_1$ and $\phi_2$ for every $2\pi$-periodic function $u^{\ast}$. Note however that $B$ only connects level 2 to levels 1 and 3. In other words, it is enough to find a $2\pi$-periodic function $u^{\ast}$ such that
$E^{(2,3)}(u^{\ast})$ is zero and $E^{(1,2)}(u^{\ast})$ is not zero (and as large as possible) to induce the desired transfer. This is achieved, for instance, with the sequence of piecewise constant controls build in \cite{Schrod2}, for which the efficiency with respect to transition $(1,2)$ tends to $\cos(\pi/6)$
and the efficiency with respect to transition $(2,3)$ is zero. Another example is presented on Figure 3.

 \begin{figure}
  \centering
 \includegraphics[width=8.6cm]{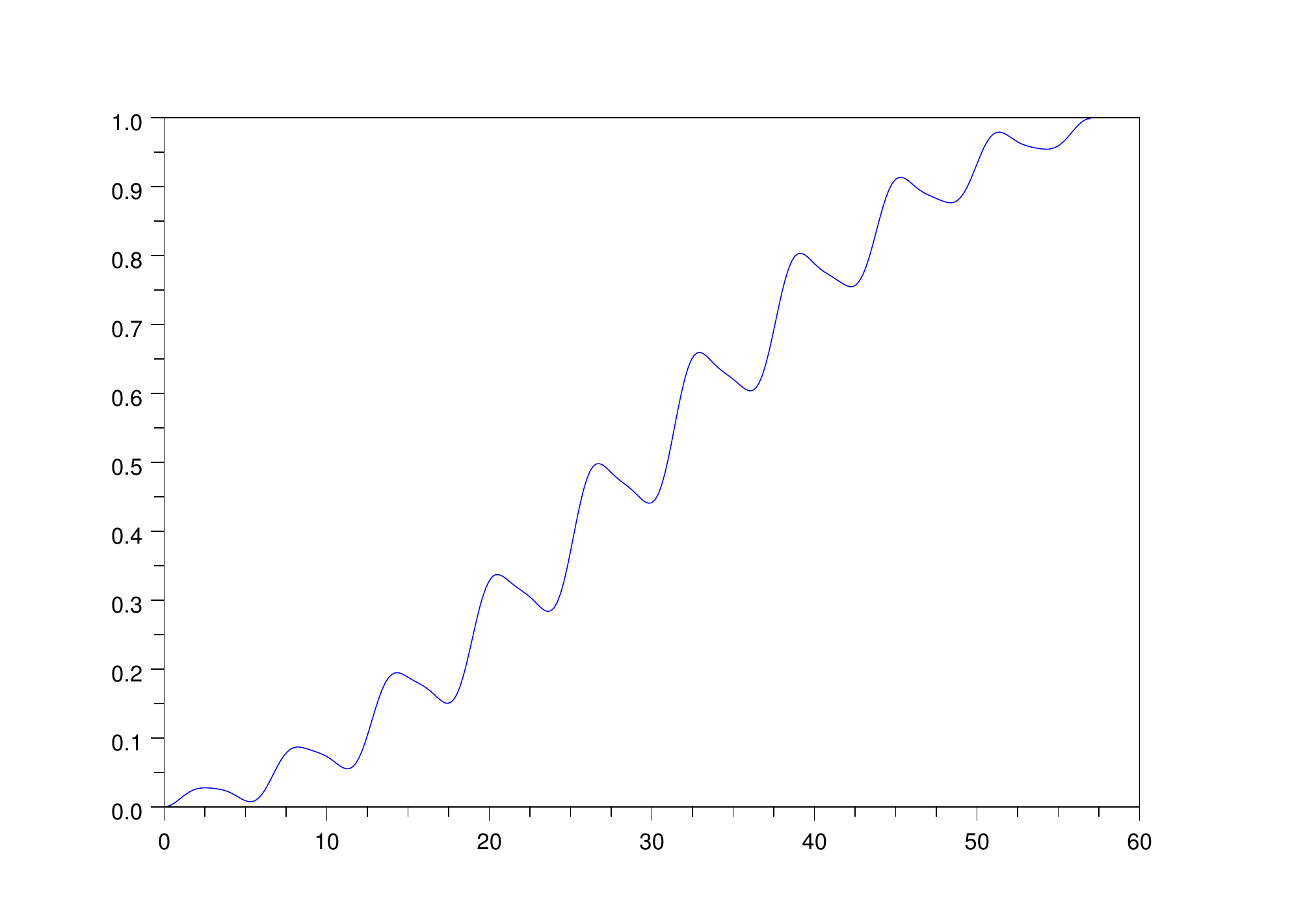}
\caption{Evolution of the square of the modulus of the second coordinate when applying the control $\frac{3}{40} \cos(t) +\frac{1}{10}$ on the planar molecule (even subspace) with initial condition $\phi_1$. The simulation has been done on a Galerkin approximation of size $N=22$. Precision $1-p^{\dag}$ is equal to $2. 10^{-3}$. Numerical efficiencies are  $38\%$ (theoretical: $3/8$) for the transition $(1,2)$ and less than $5.10^{-4}$ for the transition $(2,3)$ (theoretical: $0$).}
 \end{figure}

\addtolength{\textheight}{-8.2cm}
%%%%%%%%%%%%%%%%%%%%%%%%%%%%%%%%%%%%%%%%%%%%%%%%%%%%%%%%%%%%%%%%%%%%%%%%%%%%%%%%
\section{CONCLUSIONS AND FUTURE WORKS}

\subsection{Conclusions}
The contribution of this paper is twofold. First, we have shown how simple regularity hypotheses can be used to approach with arbitrary precision an infinite dimensional system with its finite dimensional Galerkin approximations. Using this finite dimensional reduction, we then used classical averaging techniques to obtain a proof of a well known experimental result about periodic control laws for the bilinear Schr\"{o}dinger equation. As byproduct, we introduced the notion of efficiency, which characterizes the quality of the shape of a given control law.

\subsection{Future Works}
Most of the points in this paper are merely a starting point to further investigations. Among other, we plan to study the generalization of the notion of weakly-coupled systems 
%(and the associated Good Galerkin approximations) 
for systems with continuous or mixed spectrum.

%%%%%%%%%%%%%%%%%%%%%%%%%%%%%%%%%%%%%%%%%%%%%%%%%%%%%%%%%%%%%%%%%%%%%%%%%%%%%%%%
\section{ACKNOWLEDGMENTS}

It is a pleasure for the authors to thank Ugo Boscain, Mario Sigalotti, Chitra Rangan and Dominique Sugny for discussions and advices.

This work has been supported by the INRIA Nancy-Grand Est Color ``CUPIDSE'' program.

Second and third authors were partially supported by French Agence National de
la Recherche ANR ``GCM'', program ``BLANC-CSD'', contract number NT09-504590.
The third author was partially supported by European Research Council ERC StG
2009 ``GeCoMethods'', contract number 239748.

%%%%%%%%%%%%%%%%%%%%%%%%%%%%%%%%%%%%%%%%%%%%%%%%%%%%%%%%%%%%%%%%%%%%%%%%%%%%%%%%

\bibliographystyle{plain}
\bibliography{biblio}
%
%\begin{thebibliography}{99}
%
%\bibitem{c1}
%J.G.F. Francis, The QR Transformation I, {\it Comput. J.}, vol. 4, 1961, pp
%265-271.
%
%\bibitem{c2}
%H. Kwakernaak and R. Sivan, {\it Modern Signals and Systems}, Prentice Hall,
%Englewood Cliffs, NJ; 1991.
%
%\bibitem{c3}
%D. Boley and R. Maier, "A Parallel QR Algorithm for the Non-Symmetric Eigenvalue
%Algorithm", {\it in Third SIAM Conference on Applied Linear Algebra}, Madison,
%WI, 1988, pp. A20.
%
%\end{thebibliography}

\end{document}